\documentclass[12pt]{amsart}
\usepackage{amssymb, latexsym, amsthm} 

\newtheorem{thm}{Theorem}[section] 

\newtheorem{defn}[thm]{Definition}

\newtheorem{prop}[thm]{Proposition}

\DeclareMathOperator{\ed}{ed}
\DeclareMathOperator{\edc}{edc}

\DeclareMathOperator{\disc}{disc}
\DeclareMathOperator{\Br}{Br}
\DeclareMathOperator{\Tr}{Tr}
\DeclareMathOperator{\N}{N}
\DeclareMathOperator{\Gal}{Gal}

\newcommand\operA[2]{{\if!#2!\operatorname{#1}\else{\operatorname{#1}_{#2}^{\phantom{I}}}\fi}} 

\def\Gal{{\operatorname{Gal}}}

\newcommand{\Trace}[1][]{\if!#1!\operatorname{Tr}\else{\operatorname{Tr}_{#1}^{\phantom{I}}}\fi} 

\newcommand\book[4]{{{#1},\ {{#2}}{\if!#3!\relax\else{,\ {#3}}\fi}{\if!#4!\relax\else{,\ {#4}}\fi}.}} 

\newif\ifXY 
\XYtrue     
\ifXY

\input xy
\input xyidioms.tex
\usepackage{xy}
\xyoption{all} %
\fi 

\begin{document}
\title{$\mathbb{Z}_3\times \mathbb{Z}_3$ crossed products}
\author{Eliyahu Matzri}

\thanks{This work was supported by the U.S.-Israel Binational Science Foundation (grant no. 2010149).}
\thanks{The author also thanks Daniel Krashen and UGA for hosting him while this work was done.}

\maketitle

\begin{abstract}
Let $A$ be the generic abelian crossed product with respect to $\mathbb{Z}_3\times \mathbb{Z}_3$, in this note we show that $A$ is similar to the tensor product of 4 symbol algebras (3 of degree 9 and one of degree 3) and if $A$ is of exponent $3$ it is similar to the product of 31 symbol algebras of degree $3$. We then use \cite{RS} to prove that if $A$ is any algebra of degree $9$ then $A$ is similar to the product of $35840$ symbol algebras ($8960$ of degree $3$ and $26880$ of degree $9$) and if $A$ is of exponent $3$ it is similar to the product of $277760$ symbol algebras of degree $3$. We then show that the essential $3$-dimension of the class of $A$ is at most $6$. 
\end{abstract}


\section{Introduction}
Throughout this note we let $F$ be a field containing all necessary roots of unity, denoted $\rho_n$.
The well known Merkurjev-Suslin theorem says that: assuming $F$ contains a primitive $n$-th root of $1$, there is an isomorphism
$\psi : K_2(F)/nK_2(F)\longrightarrow \Br(F)_n$ sending the symbol $\{a,b\}$ to the symbol algebra $(a,b)_{n,F}$.
In particular the $n$-th torsion part of the Brauer group is generated by symbol algebras of degree $n$.
This means every $A\in \Br(F)_n$ is similar (denoted $\sim$) to the tensor product of symbol algebras of degree $n$.
However, their proof is not constructive. It thus raises the following questions. Let $A$ be an algebra of degree $n$ and exponent $m$. Can one explicitly write $A$ as the tensor product of degree $m$ symbol algebras? Also, what is the smallest number of factors needed to express $A$ as the tensor product of degree $m$ symbol algebras? This number is sometimes called the Merkurjev-Suslin number.
These questions turn out to be quite hard in general and not much is known. Here is a short summary of some known results.
\begin{enumerate}
	\item Every degree $2$ algebra is isomorphic to a quaternion algebra.
	\item Every degree $3$ algebra is cyclic thus if $\rho_3\in F$ it is isomorphic to a symbol algebra (Wedderburn \cite{W}).
	\item Every degree $4$ algebra of exponent $2$ is isomorphic to a product of two quaternion algebras (Albert \cite{Al}).
	\item Every degree $p^n$ symbol algebra of exponent $p^m$ is similar to the product of $p^{n-m}$ symbol algebras of degree $p^m$(Tignol \cite{T1}).
	\item Every degree $8$ algebra of exponent $2$ is similar to the product of four quaternion algebras (Tignol \cite{T2}).
	\item Every abelian crossed product with respect to $\mathbb{Z}_n\times \mathbb{Z}_2$ is similar to the product of a symbol algebra of degree $2n$ and a quterinion algebra, in particular, due to Albert \cite{Al}, every degree $4$ algebra is similar to the product of a degree $4$ symbol algebra and a quaternion algebra (Lorenz, Rowen, Reichstein, Saltman \cite{LRRS}).
	\item Every abelian crossed product with respect to $(\mathbb{Z}_2)^4$  of exponent $2$ is similar to the product of $18$ quaternion algebras (Sivatski \cite{SV}).
	\item Every $p$-algebra of degree $p^n$ and exponent $p^m$ is similar to the product of $p^n-1$ cyclic algebras of degree $p^m$ (Florence \cite{MF}).
\end{enumerate}

In this paper we prove theorems \ref{MT2} and \ref{MT3} stating:

\textit{Let $A$ be an abelian crossed product with respect to $\mathbb{Z}_3\times \mathbb{Z}_3$. Then 
\begin{enumerate}
	\item $A$ is similar to the product of $4$ symbol algebras ($3$ of degree $9$ and one of degree $3$).
	\item If $A$ is of exponent $3$ then $A$ is similar to the product of $31$ symbol algebras of degree $3$.
\end{enumerate}
}

We then use \cite{RS} to deduce the general case of an algebra of degree $9$ to get theorem \ref{MT4} stating:

\textit{
Let $A$ be an $F$-central simple algebra of degree $9$. Then 
\begin{enumerate}
	\item $A$ is similar to the product of $35840$ symbol algebras, ($8960$ of degree $3$ and $26880$ of degree $9$).
	\item If $A$ is of exponent $3$ then $A$ is similar to the product of $277760$ symbol algebras of degree $3$.
\end{enumerate}}

\newpage

\section{$\mathbb{Z}_p\times \mathbb{Z}_p$ abelian crossed products}
Let $A$ be the generic abelian crossed product with respect to $G=\mathbb{Z}_p\times \mathbb{Z}_p$ over $F$, where $p$ is an odd prime. In the notation of \cite{AS} this means:
$A=(E,G,b_1,b_2,u)=E[z_1,z_2| z_i e z_i^{-1}=\sigma_i(e); z_1^p=b_1; z_2^p=b_2; z_2z_1=uz_1z_2; b_i\in E_i= E^{<\sigma_i>}; u\in E^{\times} \  s.t. \N_{E/F}(u)=1]$ where $\Gal(E/F)=<\sigma_1,\sigma_2>\cong G$.

Let $A$ be as above.
Write $E=E_1E_2$ where $E_1=F[t_1| \ t_1^{p}=f_1\in F^{\times}]$ and $E_2=F[t_2| \ t_2^{p}=f_2\in F^{\times}]$ thus we have $z_i t_i z_i^{-1}=\sigma_i(t_i)=t_i$ and $z_1 t_2=\rho_pt_2z_1; \ z_2 t_1=\rho_pt_1z_2$.
Since $b_i\in E_i$ we can write $b_1=c_0+c_1t_1+...+c_{p-1}t_1^{p-1}; \ b_2=a_0+a_1t_2+...+a_{p-1}t_2^{p-1}$  where $a_i,c_i\in F^{\times}$.

\begin{prop}\label{1}
Define $v=e_1z_1+e_2z_2$ for $e_i\in E$. If $v\neq 0$, then $[F[v^p]:F]=p$.
\end{prop}

\begin{proof}
First we compute $vt_1t_2=(e_1z_1+e_2z_2)t_1t_2=e_1z_1t_1t_2+ e_2z_2t_1t_2=\rho_pt_1t_2e_1z_1+ \rho_pt_1t_2e_2z_2=\rho_pt_1t_2(e_1z_1+e_2z_2)=\rho_pt_1t_2v$.
Thus $v^p$ commutes with $t_1t_2$ where $v$ does not, implying $[F[v]:F[v^p]]=p$. By the definition of $v$ we have $v\notin F$. Thus $\deg(A)=p^2$ imply $[F[v]:F]\in \{p,p^2\}$. If $[F[v]:F]=p$ we get that $A$ contains the sub-algebra
generated by $t_1t_2, v$ which is a degree $p$ symbol over $F$ and by the double centralizer this will imply that $A$ is decomposable which is not true in the generic case. Thus 
$[F[v]:F]=p^2$ implying $[F[v^p]:F]=p$ and we are done.
\end{proof}

The first step we take is to find a $v$ satisfying $\Tr(v^p)=0$.
In order to achieve that we will tensor $A$ with an $F$-symbol of degree $p$.

Define $B=(E_1,\sigma_2, \frac{-c_0}{a_0})\sim (E,G,1,\frac{-c_0}{a_0},1)$. Now by \cite{AS} $ A\otimes B$ is similar to $C=(E,G,b_1,\frac{-c_0}{a_0}b_2,u)$.
Abusing notation we write $z_1,z_2$ for the new ones in $C$.

\begin{prop}
Defining $v=z_1+z_2$ in $C$ we have $\Tr(v^p)=0$.
\end{prop}

\begin{proof}
First notice that $C=\sum_{i,j=0}^{p-1} Ez_1^iz_2^j$. Thus $C_0=\{d\in C | \ \Tr(d)=0\}=E_0 +\sum_{i,j=0;(i,j)\neq(0,0)}^{p-1} Ez_1^iz_2^j$ where $E_0=\sum_{i,j=0;(i,j)\neq(0,0)}^{p-1} Ft_1^it_2^j$ is the set of trace zero elements of $E$.
Now computing we see $v^p=z_1^p+e_{p-1,1}z_1^{p-1}z_2+....+e_{1,p-1}z_2^{p-1}z_1+z_2^p=b_1+e_{p-1,1}z_1^{p-1}z_2+....+e_{1,p-1}z_2^{p-1}z_1+b_2$ where $e_{i,j}\in E$. Define $r=v^p-(b_1+b_2)$. Clearly $\Tr(r)=0$, since the powers of $z_1,z_2$ in all monomial appearing in $r$ are less then $p$ and at least one is greater than zero. Thus, $v^p=b_1+b_2+r=c_0+c_1t_1+...+c_{p-1}t_1^{p-1}+(-c_0+\frac{-c_0a_1}{a_0}t_2+...+\frac{-c_0a_{p-1}}{a_0}t_2^{p-1})+r\in C_0$, and we are done.
\end{proof}

\begin{prop}
$K\doteqdot F[t_1t_2,v^p]$ is a maximal subfield of $C$.
\end{prop}
\begin{proof}
First, notice $C$ is a division algebra of degree $p^2$. To see this assume it is not, then it is similar to a degree $p$ algebra, $D$.
Thus $A\otimes B$ is similar to $D$, which implies $A$ is isomorphic to $D\otimes B^{op}$. But then $A$ has exponent $p$ which is false.
In the proof of \ref{1} we saw that $[v^p,t_1t_2]=0$ so we are left with showing $[K:F]=p^2$. Assuming $[K:F]=p$, we have $v^p\in F[t_1t_2]$. Let $\sigma$ be a generator of $\Gal(F[t_1t_2]/F)=<\sigma>$. Clearly $z_ix=\sigma(x)z_i$ for $i=1,2$ and $x\in F[t_1t_2]$, hence $vx=\sigma(x)v$, that is $\sigma(x)=vxv^{-1}$.  In particular, $\sigma(v^p)=vv^pv^{-1}=v^p$, implying $v^p\in F$. But then $C$ contains the sub algebra $F[t_1t_2,v]$ which is an F-csa of degree $p$, thus by the double centralizer $C$ would decompose into two degree $p$ algebras. This will imply that $A$ has exponent $p$, which is false. 
\end{proof}

The next step is to make $K$ Galois. Let $T$ be the Galois closure of $F[v^p]$.
Its Galois group is a subgroup of $S_p$ so has a cyclic $p$-Sylow subgroup, define $L$ to be the fixed subfield.
Clearly $F[v^p]\otimes L$ is Galois, with group $\mathbb{Z}_p$.
Thus in $C_L$ we have $K_L$ as a maximal Galois subfield with group $\mathbb{Z}_p\times \mathbb{Z}_p$.
Now writing $C_L$ as an abelian crossed product we have $C_L=(K,G,b_1,b_2,u)$ where this time we have $\Tr(b_2)=0$. Thus we can write $K_L=L[t_1t_2,t_3 | (t_1t_2)^p=f_1f_2; t_3^p=l\in L],$ \ $b_1\in L[t_1t_2]$  and $b_2=l_1t_3+...+l_{p-1}t_3^{p-1}$.

Now we change things even more. Define $D=(f_1f_2,(-\frac{f_1f_2}{l_1})^pl^{-1})_{p^2,L}=(K_L,G,t_1t_2,-\frac{f_1f_2}{l_1}(t_3)^{-1},\rho_{p^2})$ and again by \cite{AS} we have 
$R\doteqdot C_L\otimes D=(K_L,G,t_1t_2b_1,-f_1f_2-\frac{f_1f_2l_2}{l_1}t_3-...-\frac{f_1f_2l_{p-2}}{l_1}t_3^{p-2},\rho_{p^2} u)$. 

\newpage

\section{Generic $\mathbb{Z}_3\times \mathbb{Z}_3$ abelian crossed products}

From now we specialize to $p=3$.

\begin{prop}
$R$ from the end of the previous section is a symbol algebra of degree $9$.
\end{prop}

\begin{proof}
This proof is just as in \cite{LRRS}.
Since we assume $\rho_9\in F$ it is enough to find a $9$-central element.
Notice that in $R$ we have $z_2t_1t_2=\rho_3 t_1t_2z_2$; $z_2^3=-f_1f_2-\frac{f_1f_2l_2}{l_1}t_3$ and $(t_1t_2)^3=f_1f_2$.
Thus defining $x=t_1t_2+z_2$ we get $x^3=(t_1t_2+z_2)^3=(t_1t_2)^3+z_2^3=-\frac{f_1f_2l_2}{l_1}t_3$ implying $x^9=-(\frac{f_1f_2l_2}{l_1})^3 l\in L$.
Thus $R=(l_3,-(\frac{f_1f_2l_2}{l_1})^3 l)_{9,L}$ for some $l_3\in L$ and we are done.
\end{proof}


All of the above gives the following theorem:

\begin{thm}\label{MT1}
Let $A$ be a generic abelian crossed product with respect to $\mathbb{Z}_3\times \mathbb{Z}_3$. Then after a quadratic extension $L/F$ we have 
$A_L$ is similar to $R\otimes D^{-1}\otimes B^{-1}$ where $R,D,B$ are symbols as above.
\end{thm}

In order to go down to $F$ we take corestriction. Using Rosset-Tate and the projection formula, (\cite{GT} 7.4.11 and 7.2.7), we get:

\begin{thm}\label{MT2}
Let $A$ be a generic abelian crossed product with respect to $\mathbb{Z}_3\times \mathbb{Z}_3$. Then 
$A=\sum_{i=1}^{4}C_i$ where $C_1,C_2,C_3$ are symbols of degree 9 and $C_4$ is a symbol of degree 3.
\end{thm}

\begin{proof}
One gets $C_1,C_2$ from the corestriction of $R$ using R.T. $C_3$ from the corestriction of $D$ using the projection formula and $C_4$ comes from $B$.

\end{proof}

\newpage

\section{The exponent $3$ case}
In this section we will consider the case were $\exp(A)=3$.
Notice that from \ref{MT1} $A_L \sim R\otimes D^{-1}\otimes B^{-1} = (a,b)_{9,L}\otimes(\gamma,c)_{9,L} \otimes (\alpha,\beta)_{3,L}$ where $\alpha,\beta,\gamma \in F^{\times}$ and $a,b,c\in L^{\times}$.

\begin{thm}\label{MT3}
Assume $A$ has exponent $3$, then $A$ is similar to the sum of $16$ degree $3$ symbols over a quadratic extension and $31$ degree $3$ symbols over $F$.

\end{thm}

\begin{proof}
The idea for this proof is credited to L.H. Rowen, U. Vishne and E. Matzri.
Since $\exp(A)=3$ we have $F\sim A^3\sim R^3\otimes D^{-3}\otimes B^{-3}\sim R^3\otimes D^{-3}\sim (a,b)_{3,L}\otimes(\gamma,c)_{3,L}$.
Thus we get $(a,b)_{3,L}=(\gamma,c^{-1})_{3,L}$. Now by the chain lemma for degree 3 symbols in \cite{Rost} or \cite{MV} we have $x_{1,2,3} \in L^{\times}$ such that: $$(a,b)_{3,L}=(a,x_1)_{3,L}= (x_2,x_1)_{3,L}= (x_2,x_3)_{3,L}= (\gamma,x_3)_{3,L}=(\gamma,c^{-1})_{3,L}$$ 

Now we write $$(a,\frac{b}{x_1})_{9,L}\otimes(\frac{a}{x_2},x_1)_{9,L}\otimes(x_2,\frac{x_1}{x_3})_{9,L}\otimes(\frac{x_2}{\gamma},x_3)_{9,L}\otimes(\gamma,x_3c)_{9,L}\sim (a,b)_{9,L}\otimes(\gamma,c)_{9,L}$$
Thus $A\sim (a,b)_{9,L}\otimes(\gamma,c)_{9,L} \otimes (\alpha,\beta)_{3,L}\sim (a,\frac{b}{x_1})_{9,L}\otimes(\frac{a}{x_2},x_1)_{9,L}\otimes(x_2,\frac{x_1}{x_3})_{9,L}\otimes(\frac{x_2}{\gamma},x_3)_{9,L}\otimes(\gamma,x_3c)_{9,L}\sim (a,b)_{9,L}\otimes(\gamma,c)_{9,L}\otimes (\alpha,\beta)_{3,L}$ where now all the degree $9$ symbols are of exponent $3$.
But by a theorem of Tignol, \cite{T1}, each of these symbols is similar to the product of three degree $3$ symbols. Thus we have that $A_L$ is similar to the product of $16$ degree $3$ symbols and over $F$ to the product of $31$ symbols of degree $3$ and we are done.

\end{proof}
\newpage
\section{The general case of a degree $9$ algebra}

In this section we combine the results of sections $2$ and $3$ with \cite{RS} to handle the general case of a degree $9$ algebra of exponent $9$ and $3$.
Let $A$ be a $F$-central simple algebra of degree $9$.

The first step would be to follow \cite{RS} to find a field extension $P/F$
such that $A_P$ is an abelian crossed product with respect to $\mathbb{Z}_3\times \mathbb{Z}_3$ and $[P:F]$ is prime to $3$.
The argument in \cite{RS} basically goes as follows: 
Let $K\subset A$ be a maximal subfield, i.e. $[K:F]=9$.
Now let $F\subset K \subset E$ be the normal closure of $K$ over $F$.
Since we know nothing about $K$ we have to assume $G=\Gal(E/F)=S_9$.
Let $H<G$ be a $3$-sylow subgroup and $L=E^H$, then $[L:H]=4480$.
Now extend scalars to $L$, then $KL\subset A_L$ as a maximal subfield. By Galois correspondence $KL=E^{H_1}$ for some subgroup $H_1<H$ and $[H:H_1]=[KL:L]=9$.
Since $H$ is a $3$-group we can find $H_1\triangleleft H_2\triangleleft H$ such that $[H:H_2]=3$ thus we have $L=E^H\subset E^{H_2} \subset KL=E^{H_1}\subset E$ and since $H_2 \triangleleft H$ we know the extension $E^{H_2}/L$ is Galois with group~$\Gal(E^{H_2}/L)=<\sigma>\cong H/H_2\cong C_3$.
Thus in $A_L$ we have the subfield $E^{H_2}$ which has a non trivial $L$- automorphism $\sigma$. Now let $z\in A$ be an element inducing $\sigma$ (such $z$ exists by Skolem-Noether). Consider the subfield $L[z]/L$, since $z^3$ commutes with $E^{H_2}$ and $z$ does not $[L[z]:L[z^3]]=3$.
In the best case scenario we have $L[z^3]=L$ which will imply $A_L$ decomposes into the tensor product of two symbols of degree $3$ and we are done. In the general case we will have $[L[z^3]:L]=3$. If $L[z^3]/L$ is Galois we are done since $E^{H_2}[z^3]$ will be a maximal subfield Galois over $L$ with group isomorphic to $\mathbb{Z}_3\times \mathbb{Z}_3$, but again in general this should not be the case. However we can extend scalars to make $L[z^3]/L$ Galois, in particular consider $P=L[\disc(L[z^3])]$ then, $[P:L]=2$ and $P[z^3]/P$ is Galois and we are done.
To summarize we have found an extension $P=L[\disc(L[z^3])]$ with $[P:F]=4480\cdot 2=8960$ such that $A_P$ contains a maximal subfield $PE^{H_2}[z^3]/P$ Galois over $P$ with group isomorphic to $\mathbb{Z}_3\times \mathbb{Z}_3$.
\newpage
Combining the above with the results of sections $2,3$ and using Rosset-Tate we get the following theorem.
\begin{thm}\label{MT4}
Let $A$ be an $F$-central simple algebra of degree $9$. Then 
\begin{enumerate}
	\item $A$ is similar to the product of $35840$ symbol algebras, ($8960$ of degree $3$ and $26880$ of degree $9$).
	\item If $A$ is of exponent $3$ then $A$ is similar to the product of $ 277760$ symbol algebras of degree $3$.
\end{enumerate}

\end{thm}

\section{Application to essential dimension}
In \cite{M} Merkurjev computes the essential $p$-dimension of $PGL_{p^2}$ relative to a fixed field $k$ to be $p^2+1$.
One can interprate this result as follows:
Let $F$ be a field of definition (relative to a base field $k$) for the generic division algebra of degree $p^2$. Let $E/F$ be the prime to $p$ closure of $F$. Let $l,$ $l\subset k \subset E$, be a subfield of $E$ over which $A$ is defined. Then $l/k$ has transcendece degree at least $p^2+1$ (and such $l$ exists with transcendence degree exactly $p^2+1$).
It makes sense to define the essential dimension and the essential $p$-dimension of the class of an algebra $A$ (with respect to a fixed base field $k$).
\begin{defn}
Let $A\in \Br(F)$. Define the essential dimension and the essential $p$-dimension of the class of $A$ (with respect to a fixed base field $k$) as:
$$\edc(A)=\min\{\ed(B) | B\sim A\}$$ 
$$\edc_p(A)=\min\{\ed_p(B) | B\sim A\}$$
\end{defn}
Notice that \cite{M} for p=2 gives $\ed_2(PGL_{2^2})=5$ and for $p=3$ it gives $\ed_3(PGL_{3^2})=10$.
Now assume $F$ is prime to $p$ closed. Then as proved in \cite{RS} every $F$-csa of degree $p^2$ is actually an abelian crossed product with respect to $\mathbb{Z}_p \times \mathbb{Z}_p$.
Thus, in this language, in \cite{LRRS} they prove:
\begin{thm}
Let $A$ be a generic division algebra of degree $4$, then
$\edc(A)=\edc_2(A)=4$
\end{thm}

For $p=3$ Theorem \ref{MT2} says:

\begin{thm}
Let $A$ be a generic division algebra of degree $9$, then
$\edc_3(A)\leq 6$
\end{thm}

\end{document}